\documentclass[%
onecolumn
, hidempi
]{mpi2015-cscpreprint}

\usepackage[american]{babel}

\usepackage{graphicx}
\usepackage{wrapfig}
\usepackage{amssymb}
\usepackage{amsthm}
\usepackage{sansmath}

\usepackage[mathscr]{eucal}
\usepackage{amsmath}
\usepackage{breqn}

\numberwithin{equation}{section}
\numberwithin{figure}{section}
\numberwithin{table}{section}

\usepackage[software,hardware]{mymacros}
\usepackage{import}
\usepackage{xifthen}
\usepackage{pdfpages}
\usepackage{placeins}
\usepackage{transparent}
\usepackage{cleveref}
\usepackage{sidecap}    
\usepackage{floatrow}
\usepackage{float}
\usepackage{multicol}
\usepackage{todonotes}
\usepackage{algorithm}
\usepackage[noend]{algpseudocode}

\makeatletter
\def\algbackskip{\hskip-\ALG@thistlm}
\makeatother

\usepackage{caption}
\usepackage{subcaption}
\usepackage{arydshln}

\definecolor{lightgray}{gray}{0.9}
\usepackage{color}
\definecolor{bluegreen}{rgb}{0.0, 0.87, 0.87}

\newtheorem{remark}{Remark}
\newtheorem{definition}{Definition}

\newtheorem{theorem}{Theorem}

\usepackage{mymacros}
\usepackage{todonotes}
\usepackage{enumitem}

\newcommand{\Sn}{ {\mathbb{S}_n} }   

\newtheorem{corollary}{Corollary}[section]
\newcommand{\C}{\mathbb{C}}

\renewcommand{\S}{{\mathcal S}}
\newcommand{\M}{{\mathcal M}}
\newcommand{\V}{{\mathcal V}}
\newcommand{\W}{{\mathcal W}}
\newcommand{\R}{\mathbb{R}}
\newcommand{\XWpd}{ {\mathbb{X}^{\raisebox{0.2em}{{\fontsize{3}{2}\selectfont $>$}}}} }  
\newcommand{\XWpdpd}{ {\mathbb{X}^{\raisebox{0.2em}{{\fontsize{3}{2}\selectfont $\gg$}}}} }  

\DeclareMathOperator{\diag}{diag}

\DeclareMathOperator{\rank}{rank}
\newcommand{\ie}{i.\,e.\ }

\newcommand{\SYS}{\texttt{SYS}}
\newcommand{\sys}{\texttt{sys}}
\begin{document}

\title{Parameterized Interpolation of Passive Systems}

\author[$\dagger$]{Peter Benner}
\affil[$\dagger$]{Max Planck Institute for Dynamics of Complex Technical Systems, 39106 Magdeburg, Germany.\authorcr
	\email{benner@mpi-magdeburg.mpg.de}, \orcid{0000-0003-3362-4103}
}

\author[$\ddagger$]{Pawan Goyal}
\affil[$\ddagger$]{Max Planck Institute for Dynamics of Complex Technical Systems, 39106 Magdeburg, Germany.\authorcr
	\email{goyalp@mpi-magdeburg.mpg.de}, \orcid{0000-0003-3072-7780}
}

\author[$\ast$]{Paul Van Dooren}
\affil[$\ast$]{Department of Mathematical Engineering, Universit\'e catholique de Louvain in Louvain--la--Neuve, Belgium.\authorcr
	\email{paul.vandooren@uclouvain.be}, \orcid{0000-0002-0115-9932}
}

\shortauthor{P. Benner, P. Goyal, and P. Van Dooren}
\shortdate{}



\abstract{%
We study the tangential interpolation problem for a passive transfer function in standard state-space form. 
We derive new interpolation conditions based on the computation of a deflating subspace associated with a selection of spectral zeros of a parameterized para-Hermitian transfer function. We show that this technique improves the robustness of the low order model and that it can also be applied to non-passive systems, provided they have sufficiently many spectral zeros in the open right half plane. We analyze the accuracy needed for the computation of the deflating subspace, in order to still have a passive lower order model and we derive a novel selection procedure of spectral zeros in order to obtain low order models with a small approximation error.
}

\novelty{
	\begin{itemize}
		\item We investigate an tangential interpolation problem for passive systems.
		\item We propose interpolation conditions based on a deflating subspace.
		\item We discuss a construction of passive low-order models using the deflating subspace associated with a selection of spectral zeros.
	\end{itemize}
} 
\maketitle

\section{Introduction} \label{sec:Intro}
We consider linear and finite dimensional dynamical systems that are \emph{passive}. We restrict ourselves to continuous-time systems that can be represented in standard 
state-space form with real coefficients and real inputs, outputs and states~:
\begin{equation} \label{gstatespace}
 \begin{array}{rcl} \dot x(t) & = & Ax(t) + B u(t),\ x(0)=0,\\
y(t)&=& Cx(t)+Du(t).
\end{array}
\end{equation}
Denoting real and complex $n$-vectors ($n\times m$ matrices) by $\mathbb R^n$, $\mathbb C^{n}$ ($\mathbb R^{n \times m}$, $\mathbb{C}^{n \times m}$), respectively, 
then $u:\mathbb R\to\mathbb{R}^m$,   $x:\mathbb R\to \mathbb{R}^n$,  and  $y:\mathbb R\to\mathbb{R}^m$  are vector-valued functions denoting the \emph{input}, \emph{state}, 
and \emph{output} of the system, and the coefficient matrices satisfy $A \in \mathbb{R}^{n \times n}$, $B\in \mathbb{R}^{n \times m}$, $C\in \mathbb{R}^{m \times n}$, 
and  $D\in \mathbb{R}^{m \times m}$.

Model reduction of such systems has been a major research topic for the last three decades and led to a wealth of different approaches, as illustrated in several survey volumes \cite{Ant05}, \cite{BenMS05}, \cite{BenCOW17}. One of the proposed approaches is based on tangential interpolation \cite{GVV04}, \cite{ACA05}. This technique was originally developed for arbitrary types of rational transfer functions \cite{GVV04}, but an important drawback is that some critical properties -- such
as stability or passivity -- are not easy to satisfy and require a careful selection of interpolation conditions.
It was shown in \cite{ACA05} that when using spectral zeros of a given transfer function as interpolation conditions, then one does preserve passivity in the reduced-order model,  at least for the single-input/single-output case. This was extended by Sorensen in  \cite{Sor05} to the multi-input/multi-output case by making use of deflating subspace calculations and the Kalman-Yakubovich–Popov conditions for passivity. Numerical and structure-preserving algorithms to compute reduced-order models based on this approach are suggested in \cite{BenF06}. The link between both methods was later on pointed out by Fanizza et al. \cite{FKLN07}, who also make the connection to the so-called covariance extension problem. In the present paper, we further extend the approach of Sorensen by applying it to a class of systems that are parameterized by a scalar parameter.  The new contributions of this paper are threefold~:
\begin{enumerate}
\item we show that we can apply the deflating subspace idea to a class of parameterized systems, which improves the robustness of the reduced-order system by increasing its passivity radius,
\item we derive a novel selection technique of the subset of spectral zeros used for model reduction, which attempts to minimize the approximation error, and
\item we show that the method can be applied to non-passive systems and still constructs passive lower order models, under certain conditions.
\end{enumerate}
Because of the last property, we give in this paper a new derivation of Sorensen's results in order to show that it may apply also to non-passive systems.

The notation used in the paper is as follows. The Hermitian (or conjugate) transpose (transpose) of a vector or matrix $V$ is denoted by $V^{\mathsf{H}}$ ($V^{\mathsf{T}}$) and the identity matrix is denoted by $I_n$ or $I$ if the dimension is clear. We require that input and output dimensions are equal to $m$ since we want to interpolate with passive transfer functions.  Throughout this article we will use the following notation.
We denote the set of symmetric matrices in $\mathbb{R}^{n \times n}$ by $\Sn$.
Positive definiteness (semi-definiteness) of  $M\in \Sn$ is denoted by $M\succ 0$ ($M\succeq 0$). 
In Section \ref{sec:PH}, we recall the properties of passive and of port-Hamiltonian systems in order to define the robustness measure known as the passivity radius. In Section \ref{sec:SZ} we then recall the results of Sorensen on interpolation
in spectral zeros via deflating subspace calculations. This is then extended in Section \ref{sec:Param} to the novel technique of parameterized passive interpolation. The selection technique to find appropriate spectral zeros to minimize the  
approximation error is described in Section \ref{sec:Speczeros}, and the resulting robustness property is briefly described in Section \ref{sec:Robust}. We give numerical experiments that illustrate our new method and its properties in Section \ref{sec:Numerical} and give come concluding remarks in Section \ref{sec:Conclusion}.

\section{Passive systems and port-Hamiltonian realizations} \label{sec:PH}

Passive systems are well studied in the continuous-time case.
We briefly recall some important properties following \cite{Wil72b}, and refer to the literature for a more detailed survey. We consider continuous-time systems with a rational transfer matrix $Z(s)$ and  
define the following rational matrix function of $s\in \C$~: \[ \Phi(s):=Z^{\mathsf{T}}(-s) + Z(s), \] 
which is said to be \emph{para-Hermitian} since $ \Phi(-s)^{\mathsf{T}}=\Phi(s)$. It therefore coincides with two times the Hermitian part of $Z(s)$ on the  $\imath\omega$ axis:
\[ \Phi(\imath \omega)=[Z(\imath \omega)]{^{\mathsf{H}}} + Z(\imath \omega). \] 
\begin{definition}
The rational transfer function $Z(s)$ is called \emph{strictly positive-real} \/if  $\Phi(\imath\omega)\succ 0$
for all $\omega\in \mathbb{R}$ and it is called \emph{positive-real} if $\Phi(\imath \omega)\succeq 0 $ for all $\ \omega\in \mathbb{R}$.

The transfer function $Z(s)$ is called \emph{asymptotically stable} if the eigenvalues of $A$ are in the open left half plane, and it is called 
\emph{stable} if the eigenvalues of $A$ are in the closed left half plane, with any eigenvalues occurring on the imaginary axis being semi-simple.

The transfer function $Z(s)$ is called \emph{strictly passive} \/if  it is strictly positive-real and asymptotically stable and it is called \emph{passive} if it is positive real and stable with polar residues that are Hermitian and positive semi-definite for every pole on the imaginary axis.
\end{definition}

\begin{remark}
In the classical circuit theory literature the notion of positive realness is phrased differently and it implies stability. For 
rational transfer functions it is therefore equivalent to passivity \cite{AndV06}. In this paper, though, we will use the above modified definition of
positive realness.
\end{remark}

In this paper we focus on systems that are strictly passive, which implies that the transfer matrix has no infinite or imaginary axis poles, and hence is proper.
Moreover, $\Phi(\imath\omega)\succ 0$ at $\omega=\infty$ implies that $D^{\mathsf{T}}+D\succ 0$ and that $\Phi(s)$ is regular. We will see that this restriction simplifies our discussion significantly. 
This is also a reasonable restriction because passive systems can be viewed as limiting cases of strictly passive systems. 

Since the transfer function is proper, we can represent it in standard state-space form $Z(s)=C(sI_n-A)^{-1}B+D$ and we will assume throughout the paper that this realization is 
minimal (i.e. controllable and observable).
We can associate with $\Phi(s)$ a system matrix $S(s)$ which is a generalized state-space realization of $\Phi(s)$~:
\begin{equation} \label{statespace}
{\renewcommand{\arraystretch}{1.2}
S(s) :=
\left[ \begin{array}{cc|c} 0 & A-sI_n & B \\
	A^{\mathsf{T}}+sI_n & 0 & C^{\mathsf{T}} \\ \hline B^{\mathsf{T}} & C & D^{\mathsf{T}}+D  \end{array} \right].
	}
\end{equation}
If the quadruple $\M:=\left\{A,B,C,D\right\}$ is a minimal realization of a strictly passive transfer function $Z(s)$ of McMillan degree $n$,
then $S(s)$ is a minimal realization (in generalized state-space form) of $\Phi(s)$. This transfer function has indeed degree $2n$ since 
$Z(s)$ and $Z^{\mathsf{T}}(-s)$ have no common poles because of the assumption that $Z(s)$ is asymptotically stable. Since 
$ D^{\mathsf{T}}+D$ is nonsingular, the $2n$ finite eigenvalues of the pencil $S(s)$ are then the so-called {\em spectral zeros} 
of the strictly passive transfer function $Z(s)$. 

We can apply the following congruence transformation to $S(s)$, using a symmetric matrix $X$ :
$$S_X(s)  = 	\left[ \begin{array}{cc|c} I_n & 0 & 0 \\
		-X & I_n & 0 \\ \hline 0 & 0 & I_m  \end{array} \right] \! S(s) \!
	\left[ \begin{array}{cc|c} I_n & -X & 0 \\
		0 & I_n & 0 \\ \hline 0 & 0 & I_m  \end{array} \right]  $$
\begin{equation} \label{Xtransfo}
{\renewcommand{\arraystretch}{1.2}
	= \left[ \begin{array}{cc|c} 0 & A-sI_n & B \\
		A^{\mathsf{T}}+sI_n & -A^{\mathsf{T}}X-XA & C^{\mathsf{T}}-XB \\ \hline B^{\mathsf{T}} & C-B^{\mathsf{T}}X & D^{\mathsf{T}}+D  \end{array} \right]
		}
\end{equation}
without affecting the transfer function $\Phi(s)$ of this system matrix. If the following submatrix of $S_X(s)$
\begin{equation} \label{LMI}
\W(X,\M) := \left[
\begin{array}{cc}
 - A^{\mathsf{T}}X-X\,A & C^{\mathsf{T}} - X\,B \\
C- B^{\mathsf{T}}X & D^{\mathsf{T}}+D
\end{array}
\right]
\end{equation} 
is positive semi-definite, then it can be factored as indicated below
$$ {\renewcommand{\arraystretch}{1.2}
 \W(X,\M) = 
\left[ \begin{array}{c} C_G^{\mathsf{T}} \\ D_G^{\mathsf{T}} \end{array} \right]
\left[ \begin{array}{cc} C_G &  D_G \end{array} \right],
}
$$
from which it follows that
$$ {\renewcommand{\arraystretch}{1.2}
S_X(s)  =  \left[ \begin{array}{cc|c} 0 & A-sI_n & B \\
A^{\mathsf{T}}+sI_n & C_G^{\mathsf{T}} C_G &  C_G^{\mathsf{T}}D_G \\ \hline B^{\mathsf{T}} &  
D_G^{\mathsf{T}} C_G & D_G^{\mathsf{T}}D_G  \end{array} \right]
}
$$
and hence $G(s) := C_G(s\,I_n - A)^{-1}B + D_G$ 
is the right factor of the symmetric factorization $\Phi(s)= G^{\mathsf{T}}(- s) G(s)$. This then implies that $\Phi(s)$ is positive semi-definite on the $\imath \omega$ axis.
Moreover, if $A$ is assumed to be asymptotically stable, then the matrix $X$ in \eqref{LMI} must be positive definite.
This discussion is an intuitive explanation of the following result, a proof of which can be found in \cite{Wil72b}, \cite{GeninHNSVX02}.

\begin{theorem} \label{th:LMI}
Let  $\M:=\left\{A,B,C,D\right\}$ be a minimal realization of a proper rational transfer function $Z(s)$ and let $\W(X,\M)$ be the associated matrix defined in \eqref{LMI}. Then  $Z(s)$\\
(i) is positive real if and only if there exists a real symmetric matrix $X\in \S$ such that
\begin{equation} \label{KYP0}    \W(X,\M) \succeq 0 ,
\end{equation}
(ii) is passive if and only if  there exists a real symmetric matrix $X\in \S$ such that
\begin{equation} \label{KYP1}    \W(X,\M) \succeq 0 , \quad X\succ 0,
\end{equation}
and (iii) is  strictly passive if and only if there exists a real symmetric matrix $X\in \S$ such that
\begin{equation} \label{KYP2}    \W(X,\M) \succ 0 , \quad X \succ 0.
\end{equation}
\end{theorem}
The Linear Matrix Inequality (LMI) given is \eqref{KYP1} is also known as the Kalman-Yakubovich-Popov condition for passivity. In the sequel we will make use of the solution sets of these inequalities~:
\begin{subequations}\label{LMIsolnsets}
\begin{align}
&\XWpd :=\left\{ X\in \S \left|   \W(X,\M) \succeq 0,\ X \succ 0 \right.\right\},
\label{XpdsolWpsd} \\[1mm]
&\XWpdpd :=\left\{ X\in \S \left|   \W(X,\M) \succ 0,\ X \succ 0 \right.\right\}. \label{XpdsolWpd}
\end{align}
\end{subequations}

\begin{definition}
 Every solution of the LMI \eqref{XpdsolWpsd} is called a {\em certificate} for the passivity of the model $\M$
 and every solution of the LMI \eqref{XpdsolWpd} is called a {\em certificate} for the strict passivity of the model $\M$.
\end{definition}

If $D^{\mathsf{T}}+D$ is invertible, then the solutions in $\XWpd$ where $\W(X,\M)$ is of minimum rank,
are those for which $\rank \W(X,\M) = \rank (D^{\mathsf{T}}+D)  = m$, which is the case
if and only if the Schur complement of $D^{\mathsf{T}}+D$ in $\W(X,\M)$ is zero.  This Schur
complement is associated with the continuous-time \emph{algebraic Riccati equation (ARE)}
\begin{equation} \label{riccatic}
\mathsf{Ricc}(X) := -XA-A^{\mathsf{T}}X  -(C^{\mathsf{T}}-XB)(D^{\mathsf{T}}+D)^{-1}(C-B^{\mathsf{T}}X)=0.
\end{equation}
Each symmetric solution $X$ to (\ref{riccatic}) yields a spectral factorization $\Phi(s)= G^{\mathsf{T}}(- s) G(s)$ where $G(s)$ is $m\times m$ and regular. Therefore, the spectral zeros of $\Phi(s)$ are the union of the zeros of $G(s)$ and of $G^{\mathsf{T}}(-s)$. The matrix $X$ also corresponds to an invariant subspace spanned by the columns of $U:=\left[\begin{array}{cc} I_n \\ -X \end{array}\right]$
that remains invariant under multiplication with the \emph{Hamiltonian matrix}
\begin{equation}\label{HamMatrix}
H:=\left[\begin{array}{cc} A-B (D^{\mathsf{T}}+D)^{-1} C & - B (D^{\mathsf{T}}+D)^{-1} B^{\mathsf{T}} \\
C^{\mathsf{T}} (D^{\mathsf{T}}+D)^{-1} C & -(A-B (D^{\mathsf{T}}+D)^{-1} C)^{\mathsf{T}} \end{array}\right],
\end{equation}
\ie $U$ satisfies $HU=U A_F$ for a matrix $A_{F}=A-BF$ with $F := (D^{\mathsf{T}}+D)^{-1}(C-B^{\mathsf{T}}X)$.
We point out here that the solutions $X$ of the Riccati equations are certificates for the passivity of the model $\M$, but {\em not}
for its strict passivity. We will see that this distinction plays an important role in the sequel. It is also shown in \cite{Wil72b} 
that for a minimal model $\M$, the set of solutions $X$ of the Riccati equation \eqref{riccatic} has two extremal solutions $X_-$ and $X_+$ such that
all other certificates $X\in \XWpd$ satisfy $X_- \preceq X \preceq X_+$.

\medskip

We now give a brief introduction to special realizations of passive systems, known as port-Hamiltonian system models.
\begin{definition}\label{def:ph}
A linear time-invariant \emph{port-Hamiltonian (pH) system model} of a proper transfer function, has the standard state-space form
\begin{equation} \label{pH}
 \begin{array}{rcl} \dot x  & = & (J-R)Q x + (G-P) u,\\
y&=& (G+P)^{\mathsf{T}}Qx+(N+S)u,
\end{array}
\end{equation}
where the system matrices satisfy the symmetry conditions
\begin{equation} \label{sym}
\V:= \left[ \begin{array}{cccc} -J & -G \\ G^{\mathsf{T}} & N \end{array} \right]=-\mathcal V^{\mathsf{T}},\
\W:= \left[ \begin{array}{cccc} R & P \\ P^{\mathsf{T}} & S \end{array} \right] =\mathcal W^{\mathsf{T}}\succeq 0, \  Q=Q^{\mathsf{T}} \succeq 0.
\end{equation}
\end{definition}
Port-Hamiltonian systems were introduced from a different point of view \cite{VaJ14}, but they are also known to be passive.
If the model is strictly passive then $A$ and hence also $Q$ are both invertible. We can then choose $X=Q$ as certificate to
show that the model $\M:=\left\{(J-R)Q,G-P,(G+P)^{\mathsf{T}}Q,N+S\right\}$ satisfies the KYP condition. 
Conversely, let  $\M:=\left\{A,B,C,D\right\}$ be a state-space model satisfying the KYP condition \eqref{KYP1} with a given $X\succ 0$.
Then it can always be put in port-Hamiltonian form, as indicated in \cite{BMX17}.
We can use a symmetric factorization $X= T^{\mathsf{T}}T$, which implies the invertibility of $T$, and define a new realization
\[
\{A_T,B_T,C_T,D_T\} := \{TAT^{-1}, TB, CT^{-1}, D \}
\]
so that
\begin{eqnarray}  \nonumber
&&\left[ \begin{array}{cccc} T^{-\mathsf{T}} & 0\\ 0 & I_m
\end{array}
\right]
\left[ \begin{array}{cccc} -A^{\mathsf{T}}X-XA & C^{\mathsf{T}}-XB \\ C-B^{\mathsf{T}}X & D^{\mathsf{T}}+D
\end{array}
\right]
\left[ \begin{array}{cccc} T^{-1} & 0\\ 0 & I_m
\end{array}
\right] \\
\label{PH}
&& \qquad =
\left[ \begin{array}{cccc}-A_T & -B_T \\ C_T & D_T
\end{array}
\right]+
\left[ \begin{array}{cccc} -A_T^{\mathsf{T}} & C^{\mathsf{T}}_T \\ -B^{\mathsf{T}}_T & D^{\mathsf{T}}_T
\end{array}
\right]
\succeq 0.
\end{eqnarray}
We can then use the symmetric and skew-symmetric part of the matrix
\[
 \S := \left[ \begin{array}{cccc}-A_T & -B_T \\ C_T & D_T
\end{array}
\right]
\]
to define the coefficients of a pH representation  via
\[
\V := \left[ \begin{array}{cccc} J &  G \\ -G^{\mathsf{T}} & N \end{array} \right] :=  \frac{{\S} - {\S}^{\mathsf{T}}}{2},  \quad 
\W := \left[ \begin{array}{cccc} R & P \\ P^{\mathsf{T}} & S \end{array} \right] := \frac{{\S} +{\S}^{\mathsf{T}}}{2} \succeq 0.
\]
This construction yields $Q=I_n$ because of the chosen factorization $X=T^{\mathsf{T}}T$. A system with such a representation will be called a {\em normalized} port-Hamiltonian system. This shows that proper port-Hamiltonian systems are nothing but passive systems described in an appropriate coordinate system.
On the other hand, the passivity radius of a normalized port-Hamiltonian system has good robustness properties
in terms of its so-called passivity radius, as is shown below.
\begin{definition}
The \emph{passivity radius} $\rho_\M$ of a passive model $\M:=\{A,B,C,D\}$ is the smallest perturbation $\Delta_\M:=\{\Delta_A,\Delta_B,\Delta_C,\Delta_D\}$  which can make the model $\M+\Delta_\M$ loose its passivity.
\end{definition}
Therefore, if the perturbation $\Delta_\M$ is measured by 
$$ \| \Delta_\M \| :=  \left\| \left[\begin{array}{cc} \Delta_A &\Delta_B \\ \Delta_C & \Delta_D \end{array}\right] \right\|_2  , 
\quad \mathrm{or} \quad \left\| \left[\begin{array}{cc} \Delta_A &\Delta_B \\ \Delta_C & \Delta_D \end{array}\right] \right\|_F,
$$ 
then, for a {\em given} certificate $X\in \XWpdpd$, the passivity condition $\W(X,\M+\Delta_\M)\succeq 0$ for all perturbed systems $\M+\Delta_\M$ becomes just a linear matrix inequality in $\Delta_\M$, and hence 
yields a computable lower bound for $\rho_\M$, which is called the $X$-passivity radius $\rho_\M(X)$~:
\[ \rho_\M(X):= \inf_{\Delta_\M\in \mathbb C^{n+m,n+m}}\left\{ \| \Delta_\M \| \; | \; \det \W(X,\M+\Delta_\M) = 0\right\} \le  \rho_\M. \]
It follows (see \cite{MeV19}) that $\rho_\M$ is the supremum of these lower bounds over all certificates $X\in \XWpdpd$~:
\begin{equation} \label{passive}
\rho_{\M}:= \sup_{X\in \XWpdpd} \rho_{\M}(X).
\end{equation}
The following theorem, proven in \cite{MeV19}, shows that normalized port-Hamiltonian systems have an $X$-passivity radius that is at least as good as the corresponding 
non-normalized system.
\begin{theorem} \label{optPH}
 Let $\M=\{A,B,C,D\}$ be a model of a strictly passive transfer function $Z(s)$. Then for every certificate $X\in\XWpdpd$, we can construct a 
 normalized port-Hamiltonian system $$\M_T:=\{A_T,B_T,C_T,D_T\}=\left\{J-R,G-P,(G+P)^{\mathsf{T}},N+S\right\}$$ using a factorization $X=T^{\mathsf{T}}T$.
 The $X$-passivity radii $\rho_\M(X)$ and $\rho_{\M_T}(I)$ of these two models satisfy 
 $$ \rho_\M(X) \le \rho_{\M_T}(I)=\lambda_{\min}(\W).
 $$
\end{theorem}
The optimal passivity radius for all possible models for $Z(s)$ must therefore be attained by a normalized port-Hamiltonian model. 
The following theorem indicates that there is such a normalized port-Hamiltonian system with
optimal passivity radius and that it corresponds to a certificate $X$ for a family of passive systems, parameterized by the real parameter $\xi$~:
\begin{equation}
  \M_\xi := \{A+\frac{\xi}{2}I_n, B,C,D-\frac{\xi}{2}I_m\}, \quad Z_\xi(s):=C((s-\xi/2)I_n-A)^{-1}B+(D-\xi I_m/2).
\end{equation}
\begin{theorem} \label{intervalxi}
 Let $Z(s)$ be a given strictly passive transfer function, then there exists a port-Hamiltonian system model $\M$ of $Z(s)$ with the largest possible passivity radius, and it corresponds to a {\em common} certificate $X$ for all the transfer functions $Z_\xi(s)$ that are strictly passive
 where $0 < \xi < \Xi$ and $\Xi$ is the smallest positive number such that $Z_\Xi(s)$ is passive, but not strictly passive.
\end{theorem}

It was shown in \cite{MeV19} that the calculation of $\Xi$ is a two dimensional optimization problem that can be solved efficiently.
Once the value of $\Xi$ is known, one can find a certificate $X\succ 0$ for the LMI $\W(X,\M_\Xi)\succeq 0$ using a Riccati equation
approach or the corresponding generalized eigenvalue problem 
$$ {\renewcommand{\arraystretch}{1.2}
S_\Xi(s) :=
\left[ \begin{array}{cc|c} 0 & A+\Xi I_n/2-sI_n & B \\
	A^{\mathsf{T}}+\Xi I_n/2+sI_n & 0 & C^{\mathsf{T}} \\ \hline B^{\mathsf{T}} & C & D^{\mathsf{T}}+D -\Xi I_m  \end{array} \right].
	}
$$
That certificate is then valid for the family of LMIs $\W(X,\M_\xi)\succ 0$ for $0 < \xi < \Xi$, and indicating that the
transfer functions $Z_\xi(s):=C((s-\xi/2)I_n-A)^{-1}B+(D-\xi I_m/2)$ are all {\em strictly passive}. We will use this fact later on to propose a family of interpolation conditions of low order passive transfer functions approximating a high order one. 

\section{Passive interpolation using spectral zeros} \label{sec:SZ}

We rederive here the technique developed by Sorensen \cite{Sor05} and Antoulas \cite{ACA05} for the construction of a degree $\hat n$ passive system 
$\hat Z(s)$ approximating a given passive system $Z(s)$ of McMillan degree $n\ge \hat n$, via interpolation in a set of so-called {\em spectral zeros}.  But we relax the conditions imposed on the transfer function $Z(s)$, because we will need this in the next section. Our proof is based on Sorenson's construction, but it applies also to non-passive systems  $Z(s)$. 

\begin{theorem}  \label{sorth:PR} Let $\M:=\{A,B,C,D\}$ be a minimal model of an $m\times m$ transfer function $Z(s)$ and let $S(s)$ be the system matrix of $\Phi(s)$. Assume that $D+D^{\mathsf{T}}\succ 0$ and that we are then given a basis for an $\hat n$ dimensional deflating subspace of $S(s)$  satisfying 
{\renewcommand{\arraystretch}{1.2}
\begin{equation} \label{defl}
\left[ \begin{array}{cc|c} 0 & A-sI_n & B \\
	A^{\mathsf{T}}+sI_n & 0 & C^{\mathsf{T}} \\ \hline B^{\mathsf{T}} & C & D^{\mathsf{T}}+D  \end{array} \right]
\left[ \begin{array}{c} U \\ V \\ \hline  W  \end{array} \right] = \left[ \begin{array}{c} V \\ -U \\ \hline  0  \end{array} \right]  (R-s I_{\hat n}),
\end{equation}}
where the spectrum of $R$ lies in the open right half plane. Then $\hat X:= - U^{\mathsf{T}}V$ is symmetric. If, moreover, $\hat X$ is invertible, then the reduced-order transfer function $\hat Z(s)$ of the projected system model 
$$\hat \M:=\{\hat A,\hat B,\hat C,\hat D\}=\{(U^{\mathsf{T}}V)^{-1}U^{\mathsf{T}}AV, (U^{\mathsf{T}}V)^{-1}U^{\mathsf{T}}B,CV,D\}$$
is strictly positive real, and if $\hat X$ is also positive definite, then $\hat Z(s)$ is strictly passive.
\end{theorem}
\begin{proof}
The proof uses several arguments given in \cite{Sor05} for the more restrictive problem of a strictly passive transfer function $Z(s)$.
The symmetry of $\hat X:=-U^{\mathsf{T}}V$ follows from the following equation, obtained from multiplying \eqref{defl} on the left with 
$\left[ U^{\mathsf{T}} \; V^{\mathsf{T}} \; W^{\mathsf{T}} \right]$~:
$$ \left[ \begin{array}{cc|c} U^{\mathsf{T}} \! & \! V^{\mathsf{T}} \! & \! W^{\mathsf{T}} \! \end{array} \right] \!
\left[ \begin{array}{cc|c} 0 & A-sI_n & B \\
A^{\mathsf{T}}+sI_n & 0 & C^{\mathsf{T}} \\ \hline B^{\mathsf{T}} & C & D^{\mathsf{T}}+D  \end{array} \right]
\left[ \begin{array}{c} U \\ V \\ \hline  W  \end{array} \right] = (U^{\mathsf{T}}V-V^{\mathsf{T}}U) (R-s I_{\hat n}).
$$
Since the left hand side is para-Hermitian, the right hand side must also be para-Hermitian, which implies that
$$ (U^{\mathsf{T}}V-V^{\mathsf{T}}U) (R-s I_{\hat n})= (R^{\mathsf{T}}+s I_{\hat n})(V^{\mathsf{T}}U-U^{\mathsf{T}}V),
$$
and finally, 
$$ (U^{\mathsf{T}}V-V^{\mathsf{T}}U) R + R^{\mathsf{T}}(U^{\mathsf{T}}V-V^{\mathsf{T}}U) = 0.
$$
Since $R$ has all its eigenvalues in the right half plane, the matrix $(U^{\mathsf{T}}V-V^{\mathsf{T}}U)$ must be zero, which implies that
$\hat X:=-U^{\mathsf{T}}V$ is symmetric. 
Using the different rows of \eqref{defl} one obtains
\begin{eqnarray} 
\label{A} AV+BW&=&VR \\ 
\label{B} A^{\mathsf{T}}U+C^{\mathsf{T}}W&=&-UR\\ 
\label{C} B^{\mathsf{T}}U+CV+(D^{\mathsf{T}}+D)W&=&0
\end{eqnarray}
and from equations \eqref{A}, \eqref{B} and the symmetry of $\hat X$, it follows that
\begin{equation}\label{D} 
(U^{\mathsf{T}}AV+V^{\mathsf{T}}A^{\mathsf{T}}U)+(U^{\mathsf{T}}B+V^{\mathsf{T}}C)W = 0.
\end{equation}
If $\hat X=-U^{\mathsf{T}}V$ is invertible, we can construct the reduced-order system model 
\begin{equation} \label{hatmodel}
\hat \M:=\{\hat A,\hat B,\hat C,\hat D\}=
\{(U^{\mathsf{T}}V)^{-1}U^{\mathsf{T}}AV, (U^{\mathsf{T}}V)^{-1}U^{\mathsf{T}}B,CV,D\}
\end{equation}
with transfer function $\hat Z(s)$, and it then follows from \eqref{C} and \eqref{D} that
\begin{equation} \label{WXMhat} \W(\hat X,\hat \M) := \left[ \begin{array}{cc}
 - \hat A^{\mathsf{T}}\hat X-\hat X\,\hat A & \hat C^{\mathsf{T}} - \hat X\,\hat B \\
\hat C- \hat B^{\mathsf{T}}\hat X & \hat D^{\mathsf{T}}+\hat D
\end{array} \right] = \left[ \begin{array}{c} \! -W^{\mathsf{T}} \! \\ I_{m} \end{array}  \right]
(\hat D^{\mathsf{T}}+\hat D) \left[ \begin{array}{cc} -W & I_{m} \! \end{array}  \right] \succeq 0,
\end{equation} 
since $\W(\hat X,\hat \M) \left[ \begin{array}{c} I_{\hat n} \\ W \end{array} \right]=0$ and $ \hat D^{\mathsf{T}}+\hat D \succ 0$. 
It then follows from Theorem \ref{th:LMI} that the function $\hat \Phi(s)=\hat Z^{\mathsf{T}}(-s)+\hat Z(s)$ is non-negative on the imaginary axis,
and hence that $\hat Z(s)$ is positive real. Moreover, \eqref{WXMhat} implies that 
$$  \hat \Phi(s)= \hat G^{\mathsf{T}}(-s)(D^{\mathsf{T}}+D)\hat G(s), \quad \mathrm{where} \quad G(s)=I_m-W(sI_{\hat n}-\hat A)^{-1}\hat B
$$
and has as zeros the eigenvalues of $R$ since \eqref{A} implies that $\hat A+\hat BW=R$. Therefore, $\hat \Phi(s)$ has no zeros on the imaginary axis and hence must be strictly positive real.
Finally, if $\hat X$ is positive definite, 
then $\hat A$ is also asymptotically stable, which means that $\hat Z(s)$ is strictly passive.
\end{proof}

\begin{remark} \label{sorrm:PR}
In \cite{Sor05}, Sorensen proves that if $Z(s)$ is strictly passive, then $\Phi(s)$ has $n$ spectral zeros in the right half plane
and the conditional assumptions of the above theorem always hold true for every choice of $\hat n\le n$ right plane spectral zeros,
implying that $\hat X$ is positive definite and $\hat Z(s)$ is strictly passive. 
\end{remark}

\begin{corollary}
The above equation \eqref{WXMhat} indicates that $\W(\hat X,\hat\M)$ has minimum rank, which  implies that its Schur complement is zero, and hence that $\hat X$ solves the Riccati equation
$$ - \hat X \hat A - \hat A^{\mathsf{T}}\hat X - (\hat C^{\mathsf{T}} - \hat X \hat B)(\hat D^{\mathsf{T}}+ \hat D)^{-1}(\hat C-\hat B^{\mathsf{T}}X) =0. 
$$
Moreover, the corresponding feedback matrix $F:=(\hat D^{\mathsf{T}}+ \hat D)^{-1}(\hat C-\hat B^{\mathsf{T}}X) $ is a stabilizing feedback since $\W(\hat X,\hat\M) \succeq 0$.
\end{corollary}

It was shown in \cite{FKLN07} that when $Z(s)$ is strictly passive and $R$ has distinct eigenvalues (which is the generic case), then the lower order model $\hat \M:=\{\hat A,\hat B,\hat C,\hat D\}$ constructed as in \eqref{hatmodel}, satisfies the following tangential interpolation conditions~:
$$
 Z(\lambda_j) Wr_j = \hat Z(\lambda_j) Wr_j, \;\; r^{\mathsf{T}}_j W^{\mathsf{T}} Z(-\lambda_j)=r^{\mathsf{T}}_jW^{\mathsf{T}} \hat Z(- \lambda_j), \;\;  j=1,...,\hat n, \;\; Z(\infty)=\hat Z(\infty),
$$
where $(\lambda_j,r_j), \; j=1,\ldots,\hat n$, is a set of self-conjugate (eigenvalue, eigenvector) pairs of $R$. 
When $R$ has distinct eigenvalues, this relates the method of Sorenson to the spectral zero interpolation approach of  
Antoulas \cite{ACA05,Ant05}. If some of the eigenvalues are repeated, the conditions imply also that derivatives at these points should match (see \cite{GVV04,Ant05}).

We give below a more complete (and simpler) proof of this connection, for the case where $Z(s)$ satisfies the relaxed conditions of 
Theorem \ref{sorth:PR}.
\begin{theorem}  \label{th:interpol} 
Let $\M:=\{A,B,C,D\}$ be a minimal model of an $m\times m$ transfer function $Z(s)$ and let $S(s)$ be the system matrix of $\Phi(s)$. Assume that $D+D^{\mathsf{T}}\succ 0$ and that we are then given a basis for an $\hat n$ dimensional deflating subspace of $S(s)$  satisfying 
{\renewcommand{\arraystretch}{1.2}
\begin{equation} \label{defl2}
\left[ \begin{array}{cc|c} 0 & A-sI_n & B \\
	A^{\mathsf{T}}+sI_n & 0 & C^{\mathsf{T}} \\ \hline B^{\mathsf{T}} & C & D^{\mathsf{T}}+D  \end{array} \right]
\left[ \begin{array}{c} U \\ V \\ \hline  W  \end{array} \right] = \left[ \begin{array}{c} V \\ -U \\ \hline  0  \end{array} \right]  (R-s I_{\hat n}),
\end{equation}}
where the spectrum of $R$ lies in the open right half plane and the matrix $\hat X:=-U^{\mathsf{T}}V$ is positive definite. 
Then the reduced-order transfer function $\hat Z(s)$ of the projected system 
$$\hat \M:=\{\hat A,\hat B,\hat C,\hat D\}=\{(U^{\mathsf{T}}V)^{-1}U^{\mathsf{T}}AV, (U^{\mathsf{T}}V)^{-1}U^{\mathsf{T}}B,CV,D\}$$
is strictly positive real, and it satisfies the following tangential interpolation conditions that define $\hat Z(s)$ completely~:
\begin{equation} \label{tanginterpol} Z(\lambda_j) Wr_j = \hat Z(\lambda_j) Wr_j, \;\; r^{\mathsf{T}}_j W^{\mathsf{T}} Z(-\lambda_j)=r^{\mathsf{T}}_jW^{\mathsf{T}} \hat Z(- \lambda_j), \;  j=1,...,\hat n, \;\; Z(\infty)=\hat Z(\infty),\end{equation}
where $(\lambda_j,r_j), \; j=1,\ldots,\hat n$, is a set of self-conjugate (eigenvalue, eigenvector) pairs of $R$. \end{theorem}
\begin{proof}
When multiplying the columns of \eqref{defl} with $r_j$, and evaluating this at $\lambda_j$, we obtain
{\renewcommand{\arraystretch}{1.2}
\[\left[ \begin{array}{cc|c} 0 & A-\lambda_j I_n & B \\
	A^{\mathsf{T}}+\lambda_j I_n & 0 & C^{\mathsf{T}} \\ \hline B^{\mathsf{T}} & C & D^{\mathsf{T}}+D  \end{array} \right]
\left[ \begin{array}{c} Ur_j \\ Vr_j \\ \hline  Wr_j  \end{array} \right] = 0,
\]}
which implies that $Wr_j$ is in the kernel of the Schur complement of the system matrix on the left~: 
\begin{equation} \label{Phi} 
 \Phi(\lambda_j) Wr_j  = \left( Z^{\mathsf{T}}(-\lambda_j)+ Z(\lambda_j) \right) Wr_j =0.
\end{equation}
It follows also from \eqref{defl} that the projected system matrix
{\renewcommand{\arraystretch}{1.2} $$
\left[ \begin{array}{cc|c} 0 & \hat A-s I_n & \hat B \\
	\hat A^{\mathsf{T}}+s I_n & 0 & \hat C^{\mathsf{T}} \\ \hline \hat B^{\mathsf{T}} & \hat C & \hat D^{\mathsf{T}}+ \hat D  \end{array} \right] :=  $$
$$
\left[ \begin{array}{cc|c} \! (U^{\mathsf{T}}V)^{-1}U^{\mathsf{T}} \! & & \\ & V^{\mathsf{T}} & \\ \hline & &  I_m  \end{array} \right] \! 
\left[ \begin{array}{cc|c} 0 & \!  A-s I_n \! & B \\
	\! A^{\mathsf{T}}+s I_n \! & 0 & C^{\mathsf{T}} \\ \hline B^{\mathsf{T}} & C & D^{\mathsf{T}}+D  \end{array} \right]
\! \left[ \begin{array}{cc|c} \! U(V^{\mathsf{T}}U)^{-1} \! & & \\ & V & \\ \hline & & I_m \end{array} \right]
$$}
has $\hat \Phi(s):= \hat Z^{\mathsf{T}}(-s)+ \hat Z(s)$ as Schur complement. Since we have 
$$ \left[ \begin{array}{cc|c} U(V^{\mathsf{T}}U)^{-1} & & \\ & V & \\ \hline & & I_m \end{array} \right]
\left[ \begin{array}{c} (V^{\mathsf{T}}U)r_j \\ r_j \\ \hline Wr_j \end{array} \right]
= \left[ \begin{array}{c} Ur_j \\ Vr_j \\ \hline Wr_j \end{array} \right]
$$
it follows that 
{\renewcommand{\arraystretch}{1.2}$$ \left[ \begin{array}{cc|c} 0 & \hat A-\lambda_j I_n & \hat B \\ \hat A^{\mathsf{T}}+\lambda_j I_n & 0 & \hat C^{\mathsf{T}} \\ \hline 
\hat B^{\mathsf{T}} & \hat C & \hat D^{\mathsf{T}}+ \hat D  \end{array} \right] 
\left[ \begin{array}{c} (V^{\mathsf{T}}U)r_j \\ r_j \\ \hline Wr_j \end{array} \right]=0
$$}
which then in turn implies that
\begin{equation} \label{hatPhi} 
\hat \Phi(\lambda_j) Wr_j  = \left( \hat Z^{\mathsf{T}}(-\lambda_j)+ \hat Z(\lambda_j) \right) Wr_j =0.
\end{equation}
This shows that the spectral zeros $\lambda_j$ and corresponding zero directions $Wr_j, j=1,\ldots,\hat n$, of $\hat \Phi(s)$ are a subset of those of the original system $\Phi(s)$.
To show that this also implies \eqref{tanginterpol} we use the same reasoning as above to obtain the equations 
$$ \left[ \begin{array}{cc} A-\lambda_j I_n & B \\ C & D  \end{array} \right]
\left[ \begin{array}{c} Vr_j \\Wr_j \end{array} \right] = \left[ \begin{array}{c} 0 \\ y_j \end{array} \right] , \quad 
 \left[ \begin{array}{cc} \hat A-\lambda_j I_n & \hat B \\ \hat C & D  \end{array} \right]
\left[ \begin{array}{c} r_j \\Wr_j \end{array} \right] = \left[ \begin{array}{c} 0 \\ y_j \end{array} \right],
$$
where $y_j:=(CV+DW)r_j$. This then implies that $y_j = Z(\lambda_j)Wr_j= \hat Z(\lambda_j)Wr_j$, which together with \eqref{Phi}, \eqref{hatPhi} and $\hat D=D$ yields \eqref{tanginterpol}. 
\end{proof}

\medskip

Notice that Theorem \ref{sorth:PR} constructs a reduced-order system and a corresponding certificate $\hat X$ for passivity, but not for strict passivity,
since the matrix $\W(\hat X,\hat\M)$ is positive semi-definite and singular, while we would prefer to construct a lower order model with a certificate for strict passivity.

\section{Parameterized interpolants} \label{sec:Param}

In this section we combine the results of Sections \ref{sec:PH} and \ref{sec:SZ} to propose a set of parameterized interpolants that have the property that the 
interpolants have a realization that is port-Hamiltonian and at the same time a passivity radius that has a sufficiently large lower bound.

For this, we proceed as follows. Let $Z(s)$ be a strictly passive transfer function of McMillan degree $n$, and suppose we are given a minimal model 
$\M:=\{A,B,C,D\}$ of $Z(s)$. We will then construct a lower order model via the spectral zeros method explained in Section \ref{sec:SZ} but applied to 
a so-called {\em shifted} transfer function~:
\begin{equation} \label{shiftedtf}
Z_\xi(s):= Z(s-\frac{\xi}{2}) -\frac{\xi}{2}I_m, \quad \mathrm{with \; model} \quad  \M_\xi:=\{A+\frac{\xi}{2}I_n,B,C,D-\frac{\xi}{2}I_m\},
\end{equation}
where $\xi$ is chosen in the open interval $(0,\Xi)$ of strictly passive systems $Z_\xi(s)$ (see Theorem \ref{intervalxi}).  
We then solve the tangential interpolation problem to produce a lower order model $\hat Z_\xi(s)$ 
of degree $\hat n < n$ using $\hat Z_\xi(\infty)=Z_\xi(\infty)=D-\frac{\xi}{2}I_m$ as well as interpolation conditions on a subset of the spectral zeros of 
$Z_\xi(s)$~:
\begin{equation} \label{shiftinterpol}
 Z_\xi(\sigma_j) W_\xi r_j = \hat Z_\xi(\sigma_j) W_\xi r_j, \quad r^{\mathsf{T}}_j W_\xi^{\mathsf{T}} Z_\xi(-\sigma_j)=r^{\mathsf{T}}_jW_\xi^{\mathsf{T}} \hat Z_\xi(-\sigma_j), 
 \quad  j=1,\ldots,\hat n, 
\end{equation} 
where $(\sigma_j,r_j), \; j=1,\ldots,\hat n$, are self-conjugate (eigenvalue, eigenvector) pairs of the matrix $R_\xi$, chosen to have its spectrum in the open right half plane,
and which is obtained from the deflating subspace equation
{\renewcommand{\arraystretch}{1.2}
\begin{equation} \label{shifted} 
\left[ \begin{array}{cc|c} 0 & \! A+\frac{\xi}{2}I_n-sI_n \! & B \\
	\! A^{\mathsf{T}}+\frac{\xi}{2}I_n+sI_n \! & 0 & C^{\mathsf{T}} \\ \hline B^{\mathsf{T}} & C & D^{\mathsf{T}}+D- \xi I_m  \end{array} \right]
\left[ \begin{array}{c} U_\xi \\ V_\xi \\ \hline W_\xi  \end{array} \right] \! = \! \left[ \begin{array}{c} V_\xi \\ \! -U_\xi \\ \hline  0  \end{array} \right]  (R_\xi-s I_{\hat n}).
\end{equation}
}
It follows that these $2m\hat n+m^2$ real conditions completely define the reduced-order model $\hat Z_\xi(s)$ and from Section \ref{sec:SZ}
that a realization of the reduced-order model is given by the quadruple
$$ \{(U_\xi^{\mathsf{T}}V_\xi)^{-1}U_\xi^{\mathsf{T}}(A+\frac{\xi}{2}I_n)V_\xi, (U_\xi^{\mathsf{T}}V_\xi)^{-1}U_\xi^{\mathsf{T}}B,CV_\xi,D-\frac{\xi}{2}I_m\}$$
which can also be written as
$$ \{(U_\xi^{\mathsf{T}}V_\xi)^{-1}U_\xi^{\mathsf{T}}AV_\xi+\frac{\xi}{2}I_{\hat n}, (U_\xi^{\mathsf{T}}V_\xi)^{-1}U_\xi^{\mathsf{T}}B,CV_\xi,D-\frac{\xi}{2}I_m\}.$$
We now rewrite these conditions in terms of the original matrix $Z(s)$ and its approximation $\hat Z(s)$ derived via this implicit shift technique.
\begin{theorem} \label{passiveproject}
 Let $\M:=\{A,B,C,D\}$ be a minimal state-space realization of a strictly passive transfer function $Z(s)$ of McMillan degree $n$, and let
 $$ \Xi := \sup_\xi \{ \xi \; | \; Z_\xi(s) \; \mathrm{is \; strictly \; passive} \}. $$ 
 Then for any $\xi\in(0,\Xi)$, we consider an $\hat n$ dimensional deflating subspace of the shifted pencil \eqref{shifted}
corresponding to the spectrum of a real matrix $R_\xi$ with eigenvalues in the right half plane. Then the matrices $U_\xi$ and $V_\xi$ have full column rank $\hat n$, 
the matrix $\hat X_\xi:= -U_\xi^{\mathsf{T}}V_\xi$ is symmetric and positive definite, and the low order transfer function $\hat Z(s)$ with model parameters
\begin{equation} \label{hatM} 
\hat \M:=\{\hat A,\hat B,\hat C,\hat D\}=
\{(U_\xi^{\mathsf{T}}V_\xi)^{-1}U_\xi^{\mathsf{T}}AV_\xi, (U_\xi^{\mathsf{T}}V_\xi)^{-1}U_\xi^{\mathsf{T}}B,CV_\xi,D\}
\end{equation}
satisfies the interpolation conditions $Z(\infty)=\hat Z(\infty)=D$ and for $\; j=1,...,\hat n$~:
\begin{equation} \label{interpol}
 Z(\sigma_j-\xi/2) W_\xi r_j = \hat Z(\sigma_j-\xi/2) W_\xi r_j, \quad
 r^{\mathsf{T}}_j W_\xi^{\mathsf{T}} Z(- \sigma_j-\xi/2)=r^{\mathsf{T}}_jW_\xi^{\mathsf{T}} \hat Z(- \sigma_j-\xi/2), 
\end{equation} 
where $(\sigma_j,r_j), \; j=1,...,\hat n$, are self-conjugate (eigenvalue, eigenvector) pairs of the matrix $R_\xi$.
Moreover, the matrix $\hat X_\xi$ is a certificate for the LMI
$$ \W(\hat X_\xi,\hat\M) := \left[ \begin{array}{cc} - \hat X_\xi \hat A - \hat A^{\mathsf{T}}\hat X_\xi & \hat C^{\mathsf{T}} - \hat X_\xi \hat B  \\
\hat C-\hat B^{\mathsf{T}}\hat X_\xi & \hat D^{\mathsf{T}}+\hat D\end{array} \right] \succeq  \xi .\diag(\hat X_\xi, I_m) \succ 0
$$
and $\xi/2$ is a lower bound for the passivity radius of the normalized port-Hamiltonian realization 
$\hat \M_{T_\xi}:=\{\hat J-\hat R,\hat G-\hat P,(\hat G+\hat P)^{\mathsf{T}},\hat N+\hat S\}$
obtained using $\hat X_\xi =T_\xi^{\mathsf{T}}T_\xi$ via the state-space transformation
\begin{equation}  \label{PHhat}
\left[ \begin{array}{cccc} T_\xi & 0\\ 0 & I_m
\end{array}
\right]
\left[ \begin{array}{cccc} \hat A & \hat B \\ \hat C & \hat D
\end{array}
\right]
\left[ \begin{array}{cccc} T_\xi^{-1} & 0\\ 0 & I_m
\end{array}
\right] =  \left[ \begin{array}{cc} \hat J- \hat R & \hat G - \hat P \\ (\hat G +\hat P)^{\mathsf{T}} & \hat N + \hat S \end{array} \right].
\end{equation}
\end{theorem}
\begin{proof}
It follows from the strict passivity of $Z_\xi(s)$ for any $\xi$ in the open interval $(0,\Xi)$ that $\hat Z_\xi(s)$ constructed using \eqref{shifted} and \eqref{hatM}, 
satisfies the conditions of Theorem \ref{sorth:PR} and Remark \ref{sorrm:PR}. Therefore, the matrix $\hat X_\xi:= -U_\xi^{\mathsf{T}}V_\xi$ is 
symmetric and positive definite. It then follows that the projected system
$$ \hat \M_\xi := \{(U_\xi^{\mathsf{T}}V_\xi)^{-1}U_\xi^{\mathsf{T}} A_\xi V_\xi, (U_\xi^{\mathsf{T}}V_\xi)^{-1}U_\xi^{\mathsf{T}}B_\xi,C_\xi V_\xi,D_\xi\}
$$
satisfies Theorem \ref{sorth:PR} with $\hat X_\xi := -U_\xi^{\mathsf{T}}V_\xi$ and hence we have 
$$ \W(\hat X_\xi,\hat \M_\xi)= \left[ \begin{array}{cc} -\hat X_\xi \hat A_\xi - \hat A^{\mathsf{T}}_\xi \hat X_\xi & \hat C^{\mathsf{T}}_\xi - \hat X_\xi \hat B_\xi \\
\hat C_\xi - \hat B^{\mathsf{T}}_\xi \hat X_\xi & \hat D_\xi + \hat D^{\mathsf{T}}_\xi \end{array} \right] \succeq 0.
$$
By using the relations between $\hat\M_\xi=\{\hat A_\xi,\hat B_\xi,\hat C_\xi,\hat D_\xi\}=\{\hat A+\frac{\xi}{2}I_n,\hat B,\hat C,\hat D-\frac{\xi}{2}I_m\}$ and $\hat\M=\{\hat A,\hat B,\hat C,\hat D\}$, we obtain the LMI
\begin{equation} \label{LMIxi} \W(\hat X_\xi,\hat \M)= \left[ \begin{array}{cc} -\hat X_\xi \hat A - \hat A^{\mathsf{T}} \hat X_\xi & \hat C^{\mathsf{T}} - \hat X_\xi \hat B \\
\hat C - \hat B^{\mathsf{T}} \hat X_\xi & \hat D + \hat D^{\mathsf{T}} \end{array} \right] \succeq 
\xi \left[ \begin{array}{cc} \hat X_\xi & 0 \\ 0 & I_m \end{array} \right] \succ 0
\end{equation}
which implies that the transformed port-Hamiltonian system \eqref{PHhat} has a passivity radius at least as large as $\frac{\xi}{2}$ since it follows from \eqref{LMIxi}
and $\hat X_\xi=T_\xi^{\mathsf{T}}T_\xi$, that 
$$ 
\frac12 \left[ \begin{array}{cccc} T_\xi^{\mathsf{-T}} & 0\\ 0 & I_m \end{array} \right] 
\W(\hat X_\xi,\hat \M)
\left[ \begin{array}{cccc} T_\xi^{-1} & 0\\ 0 & I_m \end{array} \right] 
= \left[ \begin{array}{cc} \hat R & \hat P \\ \hat P^{\mathsf{T}} & \hat S \end{array} \right] \succ \frac{\xi}{2} I_{\hat n +m}.
$$
The translation of interpolation conditions on the shifted system towards similar conditions on the original system follows directly from the 
identity \eqref{shiftedtf}.
\end{proof}

It follows from the above theorem and from Theorem \ref{optPH} that in order to have an optimal passivity radius for the reduced-order model, one should choose to put it in the normalized port-Hamiltonian form
$\{ T_\xi  \hat A  T_\xi^{-1} ,  T_\xi \hat B, \hat C  T_\xi^{-1} , \hat D\}$.

\begin{remark} \label{rem:41}
It follows from Theorem \ref{sorth:PR} that when choosing $\Xi < \xi < \lambda_{\min}(D^{\mathsf{T}}+D)$, 
the pencil $S_\xi(s)$ may still have a deflating subspace  \eqref{shifted} where $\hat X_\xi$ is positive definite, and hence yield a strictly passive reduced-order model. If this is the case, we will be able to increase the passivity radius even further.
This flexibility will be used in the section on numerical examples.
\end{remark}

\begin{remark}
 Notice that the interpolation points $\{\sigma_j,j=1,\ldots,\hat n \}$ and $\{-\sigma_j,j=1,\ldots,\hat n \}$ of the shifted system $\hat Z_\xi(s)$ 
 are mirror images of each other with respect to the origin, but this is not true anymore for the interpolation points $\{\sigma_j-\xi/2,j=1,\ldots,\hat n \}$ and $\{-\sigma_j-\xi/2,j=1,\ldots,\hat n \}$ 
 of the original system $\hat Z(s)$. 
 Moreover, since the interpolation points $\{\sigma_j,j=1,\ldots,\hat n \}$ are still in the open right half plane, 
the shifted interpolation conditions have the tendency to approximate better the transfer function in the left half plane.
\end{remark}

\section{Choosing the spectral zeros} \label{sec:Speczeros}

In this section we look at the selection of zeros and the effect of (near) non-minimality of the transfer function.
If we want to select particular spectral zeros, it is convenient to compute the individual corresponding 
eigenvectors~:
\begin{equation} \left[
 \begin{array}{cc|c} 0 & A-\lambda_j I_n & B \\  A^{\mathsf{T}}+\lambda_j I_n & 0 & C^{\mathsf{T}} \\ \hline 
B^{\mathsf{T}} &  C &  D^{\mathsf{T}}+  D  \end{array} \right] 
\left[ \begin{array}{c} Ur_j \\ Vr_j \\ \hline Wr_j \end{array} \right]=0, \quad Rr_j=\lambda_jr_j.
\end{equation}
It follows from the proof of Theorem \ref{th:interpol} that the interpolation condition becomes
\begin{equation} \label{interpolate} y_j := \left[ CVr_j +D Wr_j \right] = Z(\lambda_j) Wr_j = \hat Z(\lambda_j) Wr_j. 
\end{equation} 
Since $y_j$ and $Wr_j$ are both bounded quantities, $\lambda_j$ can not be a pole of $Z(s)$ unless
it is also a decoupling zero, implying that the system is not minimal. More formally, let $\cal N$ be the unobservable subspace of the pair $(A,C)$ then
$$   A{\cal N} \subset {\cal N}, \quad  C{\cal N}=0.
$$
This implies that $0\oplus {\cal N} \oplus 0$ is a deflating subspace of $S(s)$ with as spectrum the 
unobservable modes of the pair $(A,C)$. Choosing a vector in that deflating subspace yields 
$Ur_j=0$ and hence also a singular matrix $\hat X$. Moreover, one then has $Wr_j=0$ and the interpolation condition
\eqref{interpolate} then vanishes.
A similar reasoning on the dual system 
implies that the same problem occurs when using an uncontrollable mode of the pair $(A,B)$.
Therefore it is recommended to stay away from {\em nearly uncontrollable or unobservable modes}
when selecting spectral zeros as interpolation points. If we make sure that $\hat X$ has large eigenvalues, then 
we will stay away from non-minimality in the reduced-order model, and the interpolation conditions $\eqref{interpolate}$
will be well defined. Moreover, it makes sense to ``maximize" $\hat X$ since it is the Hamiltonian storage function
of the projected system: maximizing $\hat X$ can indeed be viewed as finding the dominant restriction of the 
Hamiltonian $X$.

\medskip

We used the following procedure to construct a nearly optimal selection of interpolating spectral zeros. Assume that
we computed the full matrix $X:=-U^{\mathsf{T}}V$, which is symmetric and positive definite. If we perform
the Cholesky decomposition with pivoting on this matrix, then the leading $\hat n \times \hat n$ submatrix $\hat X$
corresponds to a subset of $\hat n$ eigenvectors that is nearly optimal (a truly optimal selection would require to verify all possible symmetric permutations).  In practice this ``greedy" ordering of the spectral zeros works reasonably well on the examples we tried. We should point out that if one desires a {\em real} lower order model, then the pivoting strategy
should also make sure that the selected spectral zeros form a self-conjugate set, but that is easy to obtain via a post-processing of the greedy ordering : it amounts to looking for a leading subset of $\hat n$ self-conjugate spectral zeros 
in the preliminary ordered complex zeros.
We also point out that this selection procedure can also be implemented on a partial set of computed eigenvectors and 
spectral zeros, such as those one would compute using a Krylov-Schur method for large-scale problems (see e.g. \cite{BennerFS11}) combined with implicit filtering of undesired spectral zeros. Such large-scale issues, though, 
are beyond the scope of this paper.

\section{Using the robustness property} \label{sec:Robust}

It follows from Section \ref{sec:Param} that it is indicated to choose $\xi\in(0,\Xi)$ as large as possible, since this will yield an interpolant with a certificate 
for a larger passivity radius. This means that in that coordinate system we can allow for larger perturbations and still preserve passivity of the reduced-order model.
We can therefore expect to have more freedom in the numerical implementation of any algorithm computing the deflating subspace described in Theorem \ref{passiveproject}
or on the flexibility of its stopping criterion.

We first show that for a strictly passive system, there are many possibilities to construct strictly passive lower order models and that the corresponding projectors 
form an open set.

\begin{theorem}
 Let $\M:=\{A,B,C,D\}$ be a minimal state-space model for a strictly passive transfer function $Z(s)$ of McMillan degree $n$. Let $X\succ 0$ be a certificate for the 
 LMI that ensures that $Z(s)$ is strictly passive~:
\begin{equation} \label{ZLMI}
\W(X,\M)=
\left[ \begin{array}{cc} -X & 0\\ 0 & I_m \end{array} \right] 
\left[ \begin{array}{cc} A & B \\ C & D \end{array} \right] 
+ \left[ \begin{array}{cc} A^{\mathsf{T}} & C^{\mathsf{T}} \\ B^{\mathsf{T}} & D^{\mathsf{T}} \end{array} \right]
\left[\begin{array}{cc} -X & 0\\ 0 & I_m \end{array} \right] \succ 0. 
\end{equation}
If we choose any matrix $V\in \R^{n\times \hat n}$ of full column rank $\hat n$, and compute $U:=-XV\hat X^{-1}$, where $\hat X := V^{\mathsf{T}}XV$, then $U^{\mathsf{T}}V=-I_{\hat n}$ and 
the system 
\begin{equation} \label{modelhatn}
\hat \M:=\{\hat A,\hat B,\hat C,\hat D\}=\{(U^{\mathsf{T}}V)^{-1}U^{\mathsf{T}}AV,(U^{\mathsf{T}}V)^{-1}U^{\mathsf{T}}B,CV,D\}
\end{equation}
is a strictly passive lower order model of degree $\hat n$.
\end{theorem}
\begin{proof}
 It follows from \eqref{ZLMI} that 
\begin{equation} \label{LMIhatn} \left[ \begin{array}{cc} \! -V^{\mathsf{T}} X \!  & 0\\ 0 & \! I_m \! \end{array} \right] 
\left[ \begin{array}{cc} A & B \\ C & D \end{array} \right]\left[ \begin{array}{cc} V & 0\\ 0 & \! I_m \! \end{array} \right]  
+ \left[ \begin{array}{cc} \! V^{\mathsf{T}} \! & 0\\ 0 & I_m \end{array} \right]  \left[ \begin{array}{cc} A^{\mathsf{T}} & C^{\mathsf{T}} \\ B^{\mathsf{T}} & D^{\mathsf{T}} \end{array} \right]
\left[\begin{array}{cc} \! -XV \! & 0\\ 0 & \! I_m \! \end{array} \right] \succ 0. 
\end{equation}
Using $\hat X = V^{\mathsf{T}}XV$, $U\hat X=-XV$ and $U^{\mathsf{T}}V=-I_{\hat n}$, we can rewrite this as 
$$
\left[ \begin{array}{cc} -\hat X & 0\\ 0 & I_m \end{array} \right] 
\left[ \begin{array}{cc} \hat A & \hat B \\ \hat C & \hat D \end{array} \right] 
+ \left[ \begin{array}{cc} \hat A^{\mathsf{T}} & \hat C^{\mathsf{T}} \\ \hat B^{\mathsf{T}} & \hat D^{\mathsf{T}} \end{array} \right]
\left[\begin{array}{cc} -\hat X & 0\\ 0 & I_m \end{array} \right] \succ 0
$$
which proves the strict passivity of the lower order model, since $\hat X \succ 0$. 
Moreover, the matrices $U$, $V$, $\hat X$ and $U^{\mathsf{T}}V$ have full rank $\hat n$ by construction
and this is maintained in an open neighborhood of $U$ and $V$. Therefore the matrix inequality \eqref{LMIhatn}
is still valid and the constructed reduced-order models in a sufficiently small neighborhood of \eqref{modelhatn} are strictly passive.
\end{proof}

Let us suppose now that the deflating subspace described in \eqref{shifted} was inaccurate, either due to roundoff, or due to early termination of an iterative process to compute it. 
If we denote the {\em computed} quantities as $\widetilde U$, $\widetilde V$ and $\widetilde W$,
then we can construct $\widetilde R$ and residuals $\Delta_U$, $\Delta_V$ and $\Delta_W$
such that the following equation holds
{\renewcommand{\arraystretch}{1.3}
\begin{equation} \label{residual} 
\left[ \begin{array}{cc|c} 0 & A-sI_n & B \\
	A^{\mathsf{T}}+sI_n & 0 & C^{\mathsf{T}} \\ \hline B^{\mathsf{T}} & C & D^{\mathsf{T}}+D  \end{array} \right]
\left[ \begin{array}{c} \widetilde U \\ \widetilde V \\ \hline  \widetilde W  \end{array} \right] = 
\left[ \begin{array}{c} \widetilde V \\ -\widetilde U \\ \hline  0  \end{array} \right]  (\widetilde R-s I_{\hat n})
+ \left[ \begin{array}{c} \Delta_U \\ \Delta_V \\ \hline \Delta_W  \end{array} \right].
\end{equation}
}
Let us also denote the computed projected system as
$$\widetilde \M :=\{\widetilde A,\widetilde B,\widetilde C,\widetilde D\}:=
\{(\widetilde U^{\mathsf{T}}\widetilde V)^{-1}\widetilde U^{\mathsf{T}} A\widetilde V,(\widetilde U^{\mathsf{T}}\widetilde V)^{-1}\widetilde U^{\mathsf{T}} B,C\widetilde V,D\}.$$
If we define $\widetilde X:= - \widetilde U^{\mathsf{T}}\widetilde V$, then if follows from these equations that
\begin{equation} \label{Lyapunov} (\widetilde X^{\mathsf{T}} - \widetilde X)\widetilde R + \widetilde R^{\mathsf{T}}(\widetilde X^{\mathsf{T}} - \widetilde X) 
= (\Delta_U^{\mathsf{T}}\widetilde U+\Delta_V^{\mathsf{T}}\widetilde V+\Delta_W^{\mathsf{T}}\widetilde W)
- (\widetilde U^{\mathsf{T}}\Delta_U+\widetilde V^{\mathsf{T}}\Delta_V+\widetilde W^{\mathsf{T}}\Delta_W),
\end{equation}
which implies that $\widetilde X$ is nearly symmetric, and that the following matrix is nearly positive definite~:
\begin{equation} \label{perturbed} \nonumber
\left[ \begin{array}{cc} - \widetilde X \widetilde A - \widetilde A^{\mathsf{T}}\widetilde X^{\mathsf{T}}  & \widetilde C^{\mathsf{T}} - \widetilde X \widetilde B  \\ 
\widetilde C-\widetilde B^{\mathsf{T}}\widetilde X^{\mathsf{T}}  & \widetilde D^{\mathsf{T}}+\widetilde D\end{array} \right] 
 = \left[ \begin{array}{cc}  \widetilde U^{\mathsf{T}} A \widetilde V + \widetilde V^{\mathsf{T}} A^{\mathsf{T}}\widetilde U & \widetilde V^{\mathsf{T}} C^{\mathsf{T}} + \widetilde U^{\mathsf{T}} B  \\ 
C \widetilde V + B^{\mathsf{T}} \widetilde U & D^{\mathsf{T}}+ D\end{array} \right]
\end{equation}
$$ =
\left[ \begin{array}{cc} \widetilde W^{\mathsf{T}} \\ -I_m \end{array} \right] (\widetilde D^{\mathsf{T}}+\widetilde D) \left[ \begin{array}{cc} \widetilde W & -I_m \end{array} \right] + 
\left[ \begin{array}{cc} \Delta & \Delta_W \\ \Delta_W^{\mathsf{T}} & 0 \end{array} \right] \succeq 0,
$$ 
where $\Delta=(\widetilde X^{\mathsf{T}} - \widetilde X)\widetilde R + \widetilde U^{\mathsf{T}}\Delta_U+\widetilde V^{\mathsf{T}}\Delta_V -\Delta_W^{\mathsf{T}}\widetilde W$
is symmetric, because of \eqref{Lyapunov}.

Notice that this is not a valid passivity LMI since $\widetilde X$ is not symmetric. But if we replace $\widetilde X$ by its symmetric part
$\widetilde X_s=\frac12(\widetilde X+\widetilde X^{\mathsf{T}})$ then we obtain, using $\widetilde X_a=\frac12(\widetilde X-\widetilde X^{\mathsf{T}})$
$$ \left[ \begin{array}{cc} \! - \widetilde X_s \widetilde A - \widetilde A^{\mathsf{T}}\widetilde X_s \!  & \! \widetilde C^{\mathsf{T}} - \widetilde X_s \widetilde B \!  \\ 
\! \widetilde C-\widetilde B^{\mathsf{T}}\widetilde X_s \! & \widetilde D^{\mathsf{T}}+\widetilde D\end{array} \right] 
 \! =\! 
\left[ \begin{array}{cc} \widetilde W^{\mathsf{T}} \\ \! -I_m \! \end{array} \right] (\widetilde D^{\mathsf{T}}+\widetilde D) \left[ \begin{array}{cc} \! \widetilde W \! & \! -I_m \! \end{array} \right] + 
\left[ \begin{array}{cc} \Delta_{11} & \! \Delta_{12} \! \\ \Delta_{12}^{\mathsf{T}} & 0 \end{array} \right] \succeq 0,
$$
where $\Delta_{11}=\Delta -\widetilde X_a \widetilde A - \widetilde A^{\mathsf{T}} \widetilde X_a  $ and $\Delta_{12} = \Delta_W -\widetilde X_a \widetilde B $.
Notice that $\widetilde X_a$ is a solution of the Lyapunov-like equation \eqref{Lyapunov} and hence that the perturbation of the above passivity LMI is
of the order of the residual in \eqref{residual}. 

This shows that if we have a robustness margin in the unperturbed system, in the sense that its passivity radius is bounded away from 0, then strict passivity is
maintained for a reasonably large residual in \eqref{perturbed}. We can thus apply these ideas to the technique of shifted interpolation and guarantee
that the perturbations induced by the numerical algorithm do not destroy the strict passivity of the projected model.
Notice that when using iterative algorithms for large-scale problems, such robustness properties may come in handy since we may allow for  early termination of iterative schemes, provided the resulting perturbation lies within the robustness bounds.

\section{Numerical experiments} \label{sec:Numerical}
In this section, we illustrate the proposed methodology to construct passive reduced-order models by means of two numerical examples. All the experiments were conducted using \matlab2020b.

\subsection{RLC circuit:} We first illustrate the results of the parameterized interpolation technique by applying it to the $200^{th}$ order single-input/single-output model of a circuit described in \cite{GugA03}, where $100$ electrical capacitances, inductors, and resistances are interconnected. The limiting value  $\Xi\approx 0.56$ for the parameter $\xi$ was estimated using a mesh of equidistant points in the interval $[0,\Xi_{ub}]$, where $\Xi_{ub}$ is a conservative upper bound computed from the spectrum of $A$ (see \cite{MeV19}). We applied the selection procedure of spectral zeros described in Section \ref{sec:Speczeros} for lower order degrees $k=\{2,4,\ldots,20\}$ and for equidistant shifts $\xi\in [0,\Xi]$. 

\begin{figure}[tb]
	\centering
\includegraphics[width=12cm]{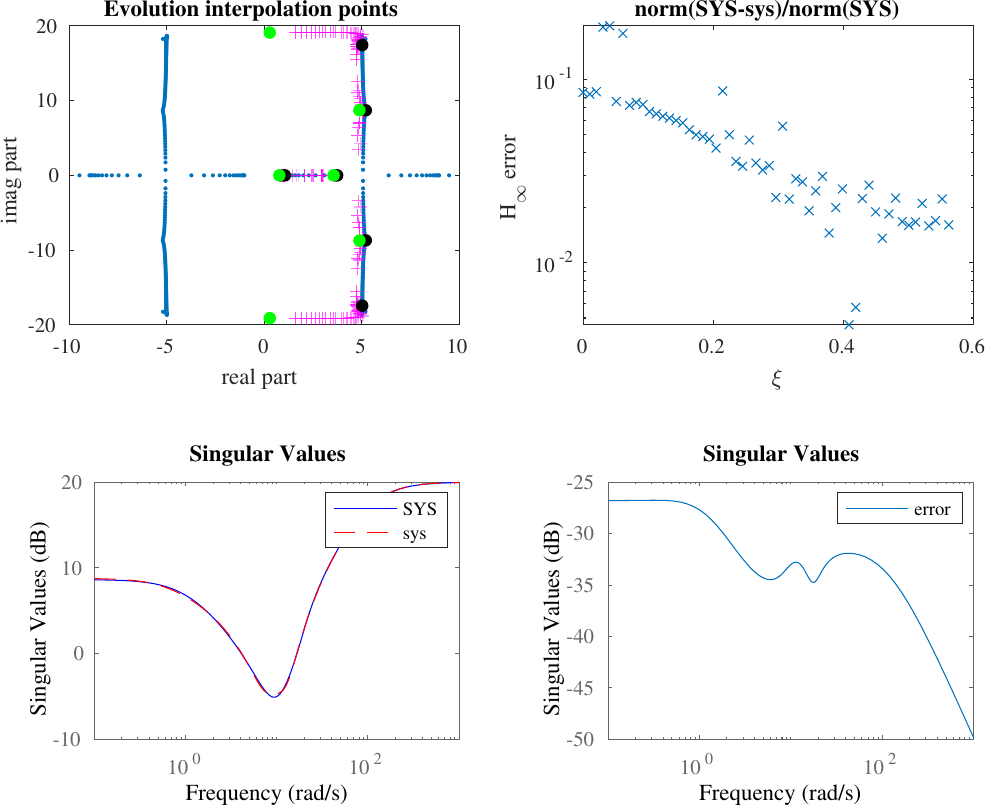}	
	\caption{Degree $6$ approximation of 200th order RLC network. Top left: spectral zeros (dots) and right half plane interpolation sets of points (+).  Black and green indicate the interpolation points corresponding to $\xi = 0$ and $\xi = \Xi$, respectively. Top right: relative $H_\infty$-error norm as a function of $\xi$. Bottom left: singular value plot of the original system (\SYS) and its best approximation (\sys). Bottom right: singular value plot of the corresponding error system.}
	\label{fig:degree6}
\end{figure}

In Figure \ref{fig:degree6}, we give the results of the low order model of degree $6$ for equidistant shifts $\xi \in [0,0.56]$. 
The top-left plot shows the original spectral zeros (in blue dots) and the selected right half plane interpolation points in magenta color for different values of $\xi$. Moreover, low-intensity magenta color $`+`$ belongs to lower values of $\xi$; likewise, high-intensity color belongs to larger values of $\xi$. One can see that the interpolation points are close to the original spectral zeros but with a shift towards the imaginary axis as $\xi$ increases. The top right plot gives the 
relative $H_\infty$-error norm as a function of $\xi$. One can see that the errors depend in a non-smooth manner on the parameter $\xi$, which is not so surprising since the selected interpolation points also depend on $\xi$. It is to be noted, though, that there is a general decreasing trend of the relative error as a function of $\xi$. This is also the case for the other low-order models we constructed. The bottom two plots give the singular value plot of the original system (\SYS) and its best approximation (\sys), and the singular value plot of the corresponding error system (\SYS-\sys), respectively.
\begin{figure}[tb]
	\centering
\includegraphics[width=8cm]{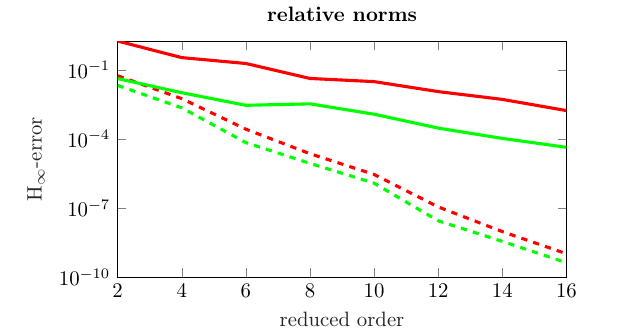}	
	\caption{Bounds for the low order error norms. The solid lines lines give upper and lower bounds for the models constructed for different values of $\xi$. The dashed lines give provable upper and lower bounds for the optimal $H_\infty$ reduced-order model.}
	\label{fig:Bounds}
\end{figure}
In order to show the effect of the order selection, we give in Figure \ref{fig:Bounds} a plot of lower and upper values 
of the achieved relative $H_\infty$-errors for the different values of $\xi$, as a function of the order $k$. In comparison, we also included provable upper and lower bounds of the relative error for the optimal $H_\infty$ approximation of the given system. Note that the upper bound is computed using the standard balanced truncation method, see, e.g., \cite{Ant05}, and the lower error bound for reduced models of order $r$ can be determined by $\sigma_{k+1}$, where $\sigma_{k+1}$ is the $(k+1)$th largest singular value of the original system \cite{morGlo84}.   It is clear from this plot that our selection procedure is far from optimal. One should be aware, though, that our
procedure is restricted to lower order systems that are passive and are generated by interpolation of special sets of points, which is a restrictive constraint.

\medskip

\subsection{Random example:} 
The second example is a random port-Hamiltonian system $\M:=\{J-R,G-P,G^{\mathsf{T}}+P^{\mathsf{T}},N+S\}$ with $X$-passivity radius  $\rho_\M(I_n)= 0.5$. The state-space model has state dimension $n=6$ and input/output dimension $m=2$. It was generated by constructing a random symmetric matrix $\W:=\left[\begin{smallmatrix} R & P \\ P^{\mathsf{T}} & S\end{smallmatrix}\right]$ with smallest eigenvalue $\lambda_{\min} =0.5$ and a random anti-symmetric matrix 
$\V:=\left[\begin{smallmatrix} -J & -G  \\ G^{\mathsf{T}} & N \end{smallmatrix}\right]$.
We then applied equidistant shifts $\xi \in [0,\Xi]$ and computed reduced-order models
of degree $\hat n=4$, based on the parameterized method explained in Theorem \ref{passiveproject}. But based on Remark \ref{rem:41}, we also took values of the shift $\xi > \Xi$, for as long as the construction of a positive definite matrix $\hat X_\xi$ was possible (which implies that $\hat A_\xi$ is still stable and $\hat D^{\mathsf{T}}_\xi+\hat D_\xi \succ 0$).

In Figure \ref{fig:nonpass}, the top-left plot shows the original spectral zeros (in blue dots) and selected right half plane interpolation points as magenta-colored $`+`$. One can see that as a function of $\xi$, the choice of four interpolation points is now much closer to each other 
than in the previous example. The top right plot gives the 
relative $H_\infty$-error norm as a function of $\xi$, and there also, one observes a smoother behavior since essentially the same interpolation points are being used. We point out here that in this plot, the blue crosses correspond to the values of
$\xi\in[0,\Xi]$, whereas the red circles, correspond to the values of $\xi > \Xi$.
The bottom two plots give the singular value plots of the original system (\SYS) and its best approximation (\sys) and the singular value plots of the corresponding error system (\SYS-\sys), respectively.
\begin{figure}[h]
	\centering
\includegraphics[width=12cm]{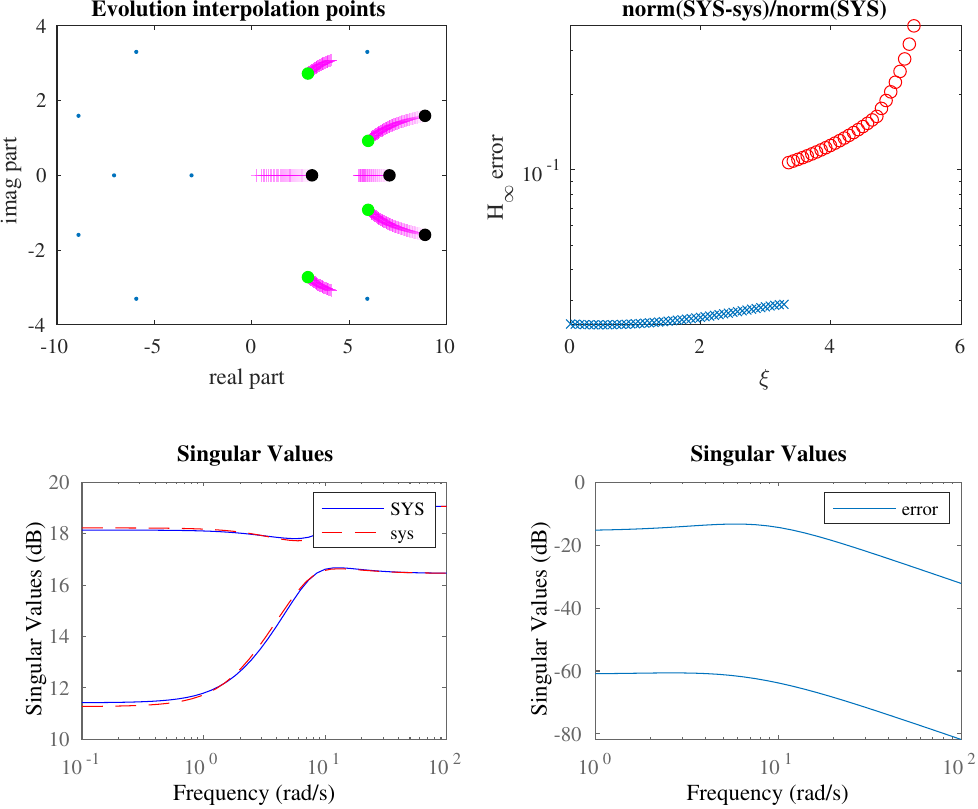}	

	\caption{Degree $4$ approximation of random $6$th order network.  Top left: spectral zeros (dots) and interpolation sets of points (crosses). Black and green indicate the interpolation points corresponding to $\xi = 0$ and $\xi \approx 1.60\Xi$, respectively. Top right: relative $H_\infty$-error norm as a function of $\xi$. Bottom left: singular value plots of the original system (\SYS) and its best approximation (\sys). Bottom right: singular value plots of the corresponding error system (\SYS-\sys).}
	\label{fig:nonpass}
\end{figure}

\section{Concluding remarks} \label{sec:Conclusion}
In this paper we developed a parameterized model reduction method based on the interpolation of the transfer function 
$Z(s)$ in a subset of the so-called spectral zeros of $Z(s)$. The parameterization lies in the fact that we consider now spectral zeros of shifted systems $Z_\xi(s)$, rather than the original transfer function $Z(s)$. Although the method is theoretically based on interpolation techniques, the algorithm itself is based on the computation of particular deflating subspaces of the ``Hamiltonian" pencils associated  with the shifted models $Z_\xi(s)$. It was also shown that these deflating subspaces do not need to be computed exactly, since the bounds on the passivity radius of the projected systems, gives a certain flexibility in the accuracy needed for the eigenspace computation. 
In this paper, we also proposed a new procedure for the selection of spectral zeros used as interpolation points for a  lower order model that is a good approximation.

\section*{Acknowledgment}
The research was performed during two visits of the third author to the Max Planck Institute in Magdeburg.

\end{document}